\documentclass[a4paper,12pt]{article}
\usepackage[cp1250]{inputenc}
\usepackage[T1]{fontenc}
\usepackage{amsfonts}
\usepackage{enumerate}
\usepackage{color}
\usepackage{amsmath}%ułatwia pisanie formuł matematycznych i ulepsza sposób ich prezentacji
\usepackage{amssymb}%pakiet od zbioru liczb rzeczywistych R
\usepackage{amsthm} % Inna część AMS-LaTeXa, pakiet do składu
\usepackage{textcomp}%pakiet do różnych znaczków
\usepackage{wrapfig}
\usepackage{url}
\usepackage{fancyhdr}%do nagłówków?
\usepackage{geometry}
\usepackage{yfonts} %do liczb gotyckich
\usepackage{stmaryrd}%pakiet do polecenia \bigtriangleup
\usepackage{bbding}
\usepackage{enumitem} %do tego aby enumerate miał wcięcia
\usepackage{bbm}%do indykatora
\everymath{\displaystyle}%indeksy pod np max
\theoremstyle{plain} % ,,Styl'' twierdzeń zwykły (wytłuszczony tytuł,
\newtheorem{tw}{Theorem}[section]	%section oznacza numerowanie w ramach rozdziałów

\newtheorem{lemma}{Lemma}[section]

\theoremstyle{definition} % Przełączamy styl na ,,definicyjny'': tytuł
\newcommand\cA{{\mathcal A}}
\newcommand\cF{{\mathcal F}}
\newcommand\cS{{\mathcal S}}
\newcommand\cM{{\mathcal M}}
\newcommand\fS{{\textfrak{S}}}

\newcommand{\rS}{\mathrm{S}}
\newcommand{\rM}{\mathrm{M}}
\newcommand{\rL}{\mathrm{L}}
\newcommand{\rP}{\mathrm{P}}

\newcommand\bIs{\mathbf{I}_{\rS}}
\newcommand\bI{\mathbf{I}}

\renewcommand\ge{\geqslant}
\renewcommand\le{\leqslant}

\newcommand{\mI}[1]{\mathbbm{1}_{#1}}
\title{An answer to an open problem \\ on seminormed fuzzy integral}
\author{Michał Boczek, Marek Kaluszka\footnote{Corresponding author. E-mail adress: kaluszka@p.lodz.pl; tel.: +48 42 6313859; fax.: +48 42 6363114.}\\ {\emph{
\small{Institute of Mathematics, Lodz University of Technology, 90-924 Lodz, Poland}}}}
\date{}
\begin{document}
\maketitle

%\begin{center}
%Michał Boczek
%\footnote{Corresponding author. E-mail adress: 800401@edu.p.lodz.pl; %tel.: +48 42 6313859; fax.: +48 42 6363114.}, Marek Kaluszka \\
%{\emph{\small{Institute of Mathematics, Lodz University of Technology, 90-924 Lodz, Poland}}}
%\end{center}

\begin{abstract} 
We give an~answer to  Problem $9.3$ stated by by Mesiar and Stup\v{n}anov\'{a} \cite{mesiar11}. We show that  
the class of semicopulas solving this problem contains any associative semicopula $\rS$ such that for each $a\in [0,1]$ the function  $x\mapsto \rS(a,x)$ is continuous and increasing on a countable number of intervals.
\end{abstract}


\section{Introduction}
Let $(X,\cA)$ be a~measurable space, where $\cA$ is a~$\sigma$-algebra of subsets of a~non-empty set $X,$ and let $\cS$ be the family of all measurable spaces. The~class of all $\cA$-measurable functions $f\colon X\to [0,1]$ is denoted by $\cF_{(X,\cA)}.$ A~{\it capacity} on $\cA$  is a~non-decreasing set function 
$\mu\colon \cA\to [0,1]$ with $\mu(\emptyset)=0$ and  $\mu(X)=1.$
We denote by $\cM_{(X,\cA)}$ the class of all 
capacities  on $\cA.$  

Suppose $\rS\colon [0,1]^2\to [0,1]$ is  a~non-decreasing function in both coordinates with the neutral element equal to $1$, called 
a~{\it semicopula} or  
a~$t$-{\it seminorm} (see $\cite{bas,dur}$).  It is clear that $\rS(x,y)\le x\wedge y$ 
and $\rS(x,0)=0=\rS(0,x)$ for all $x,y\in [0,1]$.  We denote the class of all semicopulas
by $\fS.$ Typical examples of semicopulas include: 
$\rM(a,b)=a\wedge b,$ $\Pi(a,b)=ab,$ $\rS(x,y)=xy(x\vee y)$ and $\rS _L(a,b)=(a+b-1)\vee 0;$
$\rS_L$ is called the {\it Łukasiewicz t-norm}  $\cite{klement2}.$ Hereafter, 
$a\wedge b=\min(a,b)$ and $a\vee b=\max(a,b)$.

The~generalized Sugeno integral is defined by
\begin{align*}
\bIs(\mu,f):=\sup_{t\in [0,1]} \rS\Big(t,\mu\big(\lbrace f\ge t\rbrace \big)\Big),
 \end{align*} 
where $\lbrace f\ge t\rbrace=\lbrace x\in X\colon f(x)\ge t\rbrace,$  $(X,\cA)\in \cS$ and $(\mu, f)\in \cM_{(X,\cA)}\times \cF_{(X,\cA)}.$ In the literature the functional $\bIs$ is also called {\it seminormed fuzzy integral} $\cite{suarez,klement3,ouyang3}.$
Replacing semicopula $\rS$ with $\rM$, we get the {\it Sugeno integral} $\cite{sugeno1}$. Moreover, if  $\rS=\Pi,$ then $\bI_{\Pi}$ is called the {\it Shilkret integral} $\cite{shilkret}.$

Below we present Problem $9.3$ from $\cite{mesiar11},$ which was posed by Hutník  during the conference FSTA 2014
\textit{The Twelfth International Conference on Fuzzy Set Theory and Applications} held from January 26 to January 31, 2014 in Liptovsk\'y J\'an, Slovakia. 
\medskip

{\bf Problem ${\bf 9.3}$} {\it
To characterize a~class of semicopulas $\rS$ for which the property
\begin{align}\label{r1}
\Big(\forall_{a\in [0,1]}\Big)\quad \quad \bIs\big(\mu,\rS(a,f)\big)=\rS\big(a,\bIs (\mu,f)\big)
\end{align}
holds for all $(X, \cA)\in \cS$ and all $(\mu, f)\in\cM_{(X,\cA)}\times\cF_{(X,\cA)}.$
}
\medskip

Hutník et al. $\cite{hutnik1,mesiar11}$ conjecture that the class of semicopulas solving Problem $9.3$ will contain only the (semi)copulas $\rM$ and $\Pi.$ We show that this conjecture is false, that is, the property $\eqref{r1}$ is satisfied by
any associative semicopula with continuous selections which satisfy  some mild conditions. 
 
 \section{Main results}

Let $\fS_0$ denote the set  of all  semicopulas $\rS$ which fulfill the following two conditions:
\begin{itemize}
\item[(C1)] $\rS$ is  associative, i.e. $\rS\big(\rS(x,y),z\big)=\rS(x,\rS(y,z)\big)$ for all $x,y,z\in [0,1],$
\item[(C2)]  $[0,1]\ni x\mapsto \rS(a,x)$ is continuous  for each $a\in (0,1).$
\end{itemize}
The class  $\fS_0$ is  non-empty, e.g.,  $\rM,\Pi,\rS_L\in \fS_0.$ We  prove that the property $\eqref{r1}$ implies that $\rS\in \fS_0.$

\begin{tw}\label{tw1} If the equality $\eqref{r1}$
holds for all $(X, \cA)\in \cS$ and all $(\mu, f)\in\cM_{(X,\cA)}\times\cF_{(X,\cA)},$  then $\rS\in \fS_0.$
\end{tw}

\begin{proof}
The equality $\eqref{r1}$ is obvious for $a\in \lbrace 0,1\rbrace,$ so we assume that $a\in (0, 1).$
First, we show that the property $\eqref{r1}$ implies that $\rS$ is an~associative semicopula. Indeed,  put $f=b\mI{A}$ in $\eqref{r1},$ where $b\in [0,1]$ and $A\in\cA.$ Then $\eqref{r1}$ takes the form
\begin{align*}
\sup_{t\in [0,1]}\rS\Big(t,\mu\big(A\cap\lbrace \rS(a,b)\ge t\rbrace \big)\Big)=\rS\Big(a,\sup_{t\in [0,1]} \rS\big(t,\mu\big(A\cap \lbrace b\ge t\rbrace \big)\big)\Big).
\end{align*}
Clearly, $\lbrace\rS(a,b)\ge t\rbrace=X$ and $\lbrace  b\ge t\rbrace=X$  for $t\in [0,\rS(a,b)]$ and  $t\in [0,b]$, respectively, and both the sets are empty otherwise.  Hence
\begin{align*}
\sup_{t\in [0,\rS(a,b)]}\rS\big(t,\mu(A )\big)=\rS\Big(a,\sup_{t\in [0,b]} \rS\big(t,\mu (A)\big)\Big).
\end{align*}
Since $\rS$ is non-decreasing, we get
\begin{align}\label{a1}
\rS\big(\rS(a,b),c\big)=\rS\big(a,\rS(b,c)\big)
\end{align}
for all $a\in (0,1)$ and $b,c\in [0,1].$ The equality $\eqref{a1}$ also holds for $a\in\lbrace 0,1\rbrace,$  so $\rS$ is associative.

Second, we prove that from  $\eqref{r1}$ it follows that the condition $(C2)$ holds or, equivalently, that  the following conditions are satisfied:
 \begin{enumerate}[leftmargin=1.4cm]
\item[(C2a)] $x\mapsto \rS(a,x)$ is right-continuous for each $a\in (0,1)$, 
\item[(C2b)]  $x\mapsto \rS(a,x)$ is left-continuous for each $a\in (0,1)$.
\end{enumerate}
Denote by $\rL$ and $\rP$  the left-hand side and the right-hand side of  equation $\eqref{r1},$ respectively. Let $X=[0,1]$.
Putting $\mu(A)=0$ for $A\neq X$, we get
\begin{align*}
 \rL&=\sup_{t\in[0,\,\inf_x \rS(a,f(x))]} \rS(t,1)=\inf_{x\in [0,1]} \rS\big(a,f(x)\big),\\
 \rP&=\rS\big(a,\,\sup_{z\in [0,\,\inf_{x}f(x)]} \rS(z,1)\big)=\rS(a,\inf_{x\in [0,1]} f(x)\big).
 \end{align*} 
Let $b_n\searrow b$ for a~fixed $b\in [0,1)$ and $f(x)=b_n\mI{[\frac{1}{n+1},\frac{1}{n})}(x)$ for $x\in (0,1)$ with $f(0)=f(1)=1,$ where $\mI{A}$ denotes the indicator of  $A.$ Hereafter,  $a_n\searrow a$ means that $\lim_{n\to\infty} a_n=a$ and $a_n>a_{n+1}$ for all $n.$ Since $\rL=\rP$, $\rP=\rS(a,b)$ and
\begin{align*}
\rL&=\inf_{x\in [0,1]} \rS\big(a,f(x)\big)=\lim_{n\to \infty} \rS(a,b_n),
\end{align*}
the condition $(C2a)$ holds. Now we show that $(C2b)$ is fulfilled.  Set 
$\mu(A)=1$ for all $A\neq \emptyset .$ Obviously,
\begin{align*}
 \rL=\sup_{t\in[0,\,\sup_{x}\, \rS(a,f(x))]} \rS(t,1)=\sup_{x\in [0,1]} \rS\big(a,f(x)\big),\quad\quad
 \rP=\rS\big(a,\sup_{x\in [0,1]} f(x)\big).
 \end{align*} 
Let $b_n\nearrow b$ for some $b\in (0,1],$ $f(x)=b_n\mI{[\frac{1}{n+1},\frac{1}{n})}(x)$ for 
$x\in (0,1)$ and $f(0)=f(1)=0.$  Since $\rL=\rP$, $\rP=\rS(a,b)$ and
%\begin{align*}
$\rL=\sup_{x\in [0,1]} \rS\big(a,f(x)\big)=\lim_{n\to \infty} \rS(a,b_n),
$
%\quad\end{align*}
we obtain the condition $(C2b).$
\end{proof}

Next, we show that under an additional assumption on $\rS$, the condition $\rS\in \fS_0$ is necessary and sufficient for 
\eqref{r1} to hold.  

\begin{tw}\label{tw11} Suppose that for each $a\in (0,1)$  function $x\mapsto \rS(a,x)$ is increasing on some countable number of intervals. 
Then the equality $\eqref{r1}$
holds true for all $(X, \cA)\in \cS$ and all $(\mu, f)\in\cM_{(X,\cA)}\times\cF_{(X,\cA)}$  if and only if  $\rS\in \fS_0.$
\end{tw}

To  proof  this result  we need the following lemma.

  \begin{lemma}\label{l1}
Suppose $g,h\colon [0,1]\to [0,1]$ and  $g$ is non-decreasing. 
\begin{itemize}
\item[(a)] If $g$ is right-continuous, then $g\big(\inf_{x\in [0,1]} h(x)\big)=\inf_{x\in [0,1]} g\big(h(x)\big).$
\item[(b)] If $g$ is left-continuous, then $g\big(\sup_{x\in [0,1]} h(x)\big)=\sup_{x\in [0,1]} g\big(h(x)\big).$
\end{itemize}
  \end{lemma}

\begin{proof}
\begin{itemize}
\item[(a)] Observe that  $\inf_{x\in [0,1]} g\big(h(x)\big)\ge g\big(\inf_{x\in[0,1]} h(x)\big).$ Let $(x_n)\subset [0,1]$ be such that $h(x_n)\searrow \inf_{x\in [0,1]} h(x).$  
From the right-continuity of  $g,$ we have
\begin{align*}
g\big(\inf_{x\in [0,1]}h(x)\big)=g\big(\lim_{n\to \infty} h(x_n)\big)=\lim_{n\to\infty} g\big(h(x_n)\big)\ge \inf_{x\in [0,1]} g\big(h(x)\big).
\end{align*}
Hence $g\big(\inf_{x\in [0,1]} h(x)\big)=\inf_{x\in [0,1]} g\big(h(x)\big).$

\item[(b)] Clearly, $\sup_{x\in [0,1]} g\big(h(x)\big)\le g\big(\sup_{x\in[0,1]} h(x)\big).$ Let $h(x_n)\nearrow \sup_{x\in [0,1]} h(x).$ Due to the left-continuity of  $g,$ 
\begin{align*}
g\big(\sup_{x\in [0,1]}h(x)\big)=g\big(\lim_{n\to \infty} h(x_n)\big)=\lim_{n\to\infty} g\big(h(x_n)\big)\le \sup_{x\in [0,1]} g\big(h(x)\big),
\end{align*}
thus $g\big(\sup_{x\in [0,1]} h(x)\big)=\sup_{x\in [0,1]} g\big(h(x)\big).$
\end{itemize}
\end{proof}

\begin{proof}[Proof of Theorem $\ref{tw11}$]
By $(C2)$ and Lemma $\ref{l1}\,(a)$ we get 
\begin{align*}
\rP&=\rS\Big(a,\sup_{z\in [0,1]} \rS\big(z,\mu\big(\lbrace f\ge z\rbrace \big)\big)\Big)=\sup_{z\in [0,1]} \rS\Big(a,\rS\big(z,\mu\big(\lbrace f\ge z\rbrace \big)\big)\Big).
\end{align*}
Applying $(C1)$ we obtain
\begin{align*}
\sup_{z\in [0,1]} \rS\Big(a,\rS\big(z,\mu\big(\lbrace f\ge z\rbrace \big)\big)\Big)=\sup_{z\in [0,1]} \rS\Big(\rS(a,z),\mu\big(\lbrace f\ge z\rbrace \big)\Big).
\end{align*}
In consequence, 
\begin{align}
\label{mm1}
\rP=\sup_{z\in [0,1]} \rS\Big(\rS(a,z),\mu\big(\lbrace f\ge z\rbrace \big)\Big). 
\end{align}
We prove first that $\rL=\rP$  if $x\mapsto \rS(a,x)$ is continuous and increasing in one interval,  that is, 
\[
 \rS(a,x) =
  \begin{cases}
   0, & \text{if } x\in [0,z_0] , \\
   g(x),      & \text{if } x \in [z_0,z_1],\\
   a,      & \text{if } x \in [z_1,1] ,\end{cases}\]
where $0\le z_0<z_1\le 1$ and $g\colon [z_0,z_1]\to [0,a]$ is an increasing and continuous function with $g(z_0)=0$ and $g(z_1)=a.$ Both $z_0$ and $z_1$ may depend on $a$. 
From $\eqref{mm1},$ we get
\begin{align*}
\rP&=\sup_{z\in [0,z_0]} \rS\Big(0,\mu\big(\lbrace f\ge z\rbrace \big)\Big)\vee \sup_{z\in [z_0,z_1]} \rS\Big(\rS(a,z),\mu\big(\lbrace f\ge z\rbrace \big)\Big)\vee\sup_{z\in [z_1,1]} \rS\Big(a,\mu\big(\lbrace f\ge z\rbrace \big)\Big)\\&=
 \sup_{z\in [z_0,z_1]} \rS\Big(\rS(a,z),\mu\big(\lbrace \rS(a,f)\ge \rS(a,z)\rbrace \big)\Big)\vee\sup_{z\in [z_1,1]} \rS\Big(a,\mu\big(\lbrace f\ge z\rbrace \big)\Big),
\end{align*}
since $\rS(0,y)=0$ for all $y.$ Observe that 
\begin{align*}
\sup_{z\in [z_1,1]} \rS\Big(a,\mu\big(\lbrace f\ge z\rbrace \big)\Big)&=\rS\Big(a,\mu\big(\lbrace f\ge z_1\rbrace \big)\Big)= \rS\Big(a,\mu\big(\lbrace \rS(a,f)\ge \rS(a,z_1)\rbrace \big)\Big)\\&=\rS\Big(\rS(a,z_1),\mu\big(\lbrace \rS(a,f)\ge \rS(a,z_1)\rbrace \big)\Big),
\end{align*}
as $\rS(a,z_1)=a.$ From  $\rS\big(a,[z_0,z_1]\big)=[0,a]$ we conclude that
\begin{align*}
\rP&=\sup_{z\in[z_0,z_1]} \rS\Big(\rS(a,z),\mu\big(\lbrace \rS(a,f)\ge \rS(a,z)\rbrace \big)\Big)\\&= \sup_{t\in [0,a]} \rS\Big(t,\mu\big(\lbrace \rS(a,f)\ge t\rbrace \big)\Big)=\rL.
\end{align*}
Now, we generalize the above case to a~countable number of intervals, i.e.
\[
 \rS(a,x) =
  \begin{cases}
  t_k, & \text{if } x\in [z_{2k-1},z_{2k}] , \\
  g_k(x),      & \text{if } x \in [z_{2k},z_{2k+1}],\\
   t_{k+1}, & \text{if } x\in [z_{2k+1},z_{2k+2}] , \\
     \end{cases}\]
   where $0=t_0\le t_1\le \ldots\le a$ with $\bigcup _{k\ge 0}[t_k,t_{k+1}]=[0,a]$,  $0=z_{-1}\le z_0< z_1\le \ldots\le 1$ and $g_k\colon [z_{2k},z_{2k+1}]\to [t_k,t_{k+1}]$ is an increasing function with $g_k(z_{2k})=t_k$ and $g_k(z_{2k+1})=t_{k+1}$ for each $k\ge 0.$ By $\eqref{mm1}$
   \begin{align*}
   \rP&=\sup_{k}\bigg[ \sup_{z\in [z_{2k},\,z_{2k+1}]} \rS\Big(\rS(a,z),\mu\big(\lbrace \rS(a,f)\ge \rS(a,z)\rbrace \big)\Big)\bigg].
   \end{align*}
Since
   \begin{align*}
    \sup_{z\in [z_{2k},\,z_{2k+1}]} \rS\Big(\rS(a,z),\mu\big(\lbrace \rS(a,f)\ge \rS(a,z)\rbrace \big)\Big)=\sup _{t\in [t_k,t_{k+1}]} \rS\Big(t,\mu\big(\lbrace S(a,f)\ge t\rbrace \big)\Big),
   \end{align*}
we get
   \begin{align*}
   \rP=\sup_{t\in [0,a]} \rS\Big(t,\mu\big(\lbrace \rS(a,f)\ge t\rbrace \big)\Big)=\rL.
   \end{align*}
   The proof is complete.
\end{proof}

Now we show that if  capacity $\mu$ in  Problem 9.3 is  continuous from below, then  the equality in $\eqref{r1}$  may also hold  for discontinuous semicopulas 
$\rS.$ Recall that  $\mu$ is said to be {\it continuous from below} if  $\lim_{n\to\infty} \mu(A_n)=\mu\big(\bigcup_{k=1}^\infty A_k\big)$ for $A_n\subset A_{n+1}.$

Denote by $\fS_1$ the class of all associative semicopulas $\rS$ such that
for any $a\in (0,1),$ the function $x\mapsto \rS(a,x)$  is increasing and left-continuous and has only isolated 
discountinuity points, i.e. for any discontinuity point $z_k(a)$ there exists an interval $\big(z_k(a),z_k(a)+\varepsilon_k(a)\big)$ with $\varepsilon_k(a)>0,$ on which function $x\mapsto \rS(a,x)$ is continuous.

\begin{tw}\label{tw2}
If $\rS\in \fS_1,$ 
then the equality $\eqref{r1}$ holds true for all $(X,\cA)\in \cS,$ all $f\in \cF_{(X,\cA)}$
and all continuous from below capacities $\mu\in\cM_{(X,\cA)}.$ 
\end{tw}
\begin{proof}
First, we assume  that  $x\mapsto \rS(a,x)$ has only one point of discontinuity, say $z=z(a),$ for each $a$.  Set $t_1=\rS\big(a,z\big)$ and $t_2=\lim_{x\searrow z}\rS\big(a,x\big).$ Clearly, $[0,a]=\rS\big(a,[0,1]\big)\cup B,$ where $B=\big(t_1,t_2\big].$ Hence 
\begin{align}\label{n1}
\rL&=\sup_{t\in [0,a]} \rS\Big(t,\mu\big(\lbrace \rS(a,f)\ge t\rbrace \big)\Big)\nonumber\\&=\sup_{t\in \rS(a,[0,1])} \rS\Big(t, \mu\big(\lbrace \rS(a,f)\ge t\rbrace \big)\Big)\vee \sup_{t\in B} \rS\Big(t, \mu\big(\lbrace \rS(a,f)\ge t\rbrace \big)\Big).
\end{align}
Observe that for all $t\in B$
\begin{align*}
\left\{ x\colon \rS\big(a,f(x)\big)\ge t\right\}=\left\{ x\colon \rS\big(a,f(x)\big)\ge t_2\right\}.
\end{align*}
Therefore
\begin{align}\label{n1aa}
\sup_{t\in B} \rS\Big(t,\mu\big(\lbrace \rS(a,f)\ge t\rbrace \big)\Big)&=\sup_{t\in B} \rS\Big(t,\mu\big(\lbrace \rS(a,f)\ge t_2\rbrace \big)\Big)\nonumber\\&=\rS\Big(t_2,\mu\big(\lbrace \rS(a,f)\ge t_2\rbrace \big)\Big).
\end{align}
Combining $\eqref{n1}$ and $\eqref{n1aa}$ yields
\begin{align}\label{n1a}
\rL&=\sup_{t\in \rS(a,[0,1])} \rS\Big(t,\mu\big(\lbrace \rS(a,f)\ge t\rbrace \big)\Big)\vee \rS\Big(t_2,\mu\big(\lbrace \rS(a,f)\ge t_2\rbrace \big)\Big).
\end{align}
Furthermore
\begin{align}\label{n1ab}
\sup_{t\in \rS(a,[0,1])} \rS\Big(t,\mu\big(\lbrace \rS(a,f)\ge t\rbrace \big)\Big)=\sup_{z\in [0,1]} \rS\Big(\rS(a,z),\mu\big(\lbrace \rS(a,f)\ge \rS(a,z)\rbrace \big)\Big).
\end{align}
From  the associativity of  $\rS$ and Lemma $\ref{l1}\,(b),$ we get
\begin{align}\label{n1b}
\sup_{z\in [0,1]} \rS\Big(\rS(a,z),\mu\big(\lbrace \rS(a,f)\ge \rS(a,z)\rbrace \big)\Big)&= \rS\Big(a,\sup_{z\in [0,1]} \rS\big(z,\mu\big(\lbrace f\ge z\rbrace \big)\big)\Big)=\rP.
\end{align}
Summing up, from $\eqref{n1a}$ and $\eqref{n1b}$ we obtain 
\begin{align*}
\rL=\rP\vee\rS\Big(t_2,\mu\big(\lbrace \rS(a,f)\ge t_2\rbrace \big)\Big).
\end{align*}
It remains to prove that $\rS\Big(t_2,\mu\big(\lbrace \rS(a,f)\ge t_2\rbrace \big)\Big)\le \rP.$  Since $z<1,$  
there exists a sequence $(u_n)$ such that $u_n\searrow t_2$ and $u_n= \rS(a,z_n)$ for a sequence $(z_n)$ such that $z_n>z$ for all $n.$ Observe that 
\begin{align*}
\liminf_{n\to \infty} \rS\Big(u_n,\mu\big(\lbrace \rS(a,f)\ge u_n\rbrace \big)\Big)&\ge \liminf_{n\to \infty} \rS\Big(t_2,\mu\big(\lbrace \rS(a,f)\ge u_n\rbrace \big)\Big)\\&= \lim_{n\to \infty} \rS\Big(t_2,\mu\big(\lbrace \rS(a,f)\ge u_n\rbrace \big)\Big),
\end{align*}
as $u_n>t_2$ and the sequence $\big(\mu\big(\lbrace \rS(a,f)\ge u_n\rbrace \big)\big)_n$ is non-decreasing.  From the left-continuity of $x\mapsto \rS(a,x)$ and the continuity from below of $\mu,$ we get   
\begin{align}\label{n2}
\liminf_{n\to \infty} \rS\Big(u_n,\mu\big(\lbrace \rS(a,f)\ge u_n\rbrace \big)\Big)\ge\rS\Big(t_2,\mu\big(\lbrace \rS(a,f)\ge t_2\rbrace \big)\Big).
\end{align}
Since $u_n\in \rS\big(a,[0,1]\big),$ we have
\begin{align}\label{n2a}
\liminf_{n\to \infty} \rS\Big(u_n,\mu\big(\lbrace \rS(a,f)\ge u_n\rbrace \big)\Big)\le \sup_{t\in \rS(a,[0,1])} \rS\Big(t,\mu\big(\lbrace \rS(a,f)\ge t\rbrace \big)\Big).
\end{align}
Combining $\eqref{n2}$ and $\eqref{n2a}$ with $\eqref{n1ab}$ and $\eqref{n1b}$ gives 
\begin{align*}
\rS\Big(t_2,\mu\big(\lbrace \rS(a,f)\ge t_2\rbrace \big)\Big)\le \sup_{t\in \rS(a,[0,1])} \rS\Big(t,\mu\big(\lbrace \rS(a,f)\ge t\rbrace \big)\Big)=\rP.
\end{align*}
Thus, we have  shown that if there exists only one point of discontinuity, then $\rL=\rP.$ The proof for the case of many isolated points proceeds in much the same way as above.  
\end{proof}

%\section{Acknowledgments} 
%This paper has been partially supported by the grant for young researchers from Lodz University of Technology, grant number $501\backslash 17$-2-2-$7154\backslash 2014.$

%\addcontentsline{toc}{section}{Literatura}

\begin{thebibliography}{99}
\bibitem[1]{bas} B. Bassan, F. Spizzichino, Relations among univariate aging, bivariate aging and dependence for exchangeable lifetimes, Journal of Multivariate Analysis 93 (2005) 313–339.
\bibitem[2]{dur} F. Durante, C. Sempi, Semicopul\ae, Kybernetika 41 (2005) 315-328. 
\bibitem[3]{suarez} F. Su\'{a}rez García, P. Gil \'Alvarez, Two families of fuzzy integrals, Fuzzy Sets and Systems 18 (1986) 67-81.
\bibitem[4]{hutnik1}J. Borzová-Molnárová, L. Hal\v{c}inová, O. Hutník,  The smallest semicopula-based universal integrals I: properties and characterizations, Fuzzy Sets and Systems (2014), 
http://dx.doi.org/10.1016/j.fss2014.09.023.
\bibitem[5]{klement2} E.P. Klement, R. Mesiar, E. Pap, Triangular norms, Kluwer Academic Publishers, Dordrecht, 2000.
\bibitem[6]{klement3} E.P. Klement, R. Mesiar, E. Pap, A~universal integral as common frame for Choquet and Sugeno integral, IEEE Transactions Fuzzy Sets and Systems 18 (2010) 178-187.
\bibitem[7]{mesiar11} R. Mesiar, A. Stup\v{n}anov\'{a}, Open problems from the 12th International Conference on Fuzzy Set Theory and Its Applications, Fuzzy Sets and Systems 261 (2015) 112-123.
\bibitem[8]{ouyang3} Y. Ouyang, R. Mesiar, On the Chebyshev type inequality for seminormed fuzzy integral, Applied Mathematics Letters 22 (2009) 1810-1815.
\bibitem[9]{shilkret} N. Shilkret, Maxitive measure and integration, Indagationes Mathematicae 33 (1971) 109-116.
\bibitem[10]{sugeno1} M. Sugeno, Theory of Fuzzy Integrals and its Applications, Ph.D. Dissertation, Tokyo Institute of Technology, 1974.

\end{thebibliography}
\end{document}